\documentclass[reqno,oneside, 11pt]{amsart}
 
 \usepackage{amsfonts}
\usepackage{amsmath}
\usepackage{amssymb}
\usepackage[applemac]{inputenc}
\usepackage{url}
\usepackage{hyperref}
\usepackage{color}

 \usepackage{url} 
 \usepackage{hyperref}

\usepackage[T1]{fontenc}
\usepackage{times,mathptm}
\usepackage{amssymb,epsfig,verbatim,xypic}
\theoremstyle{plain}

\newtheorem{theorem}{{Theorem}}[section]
\newtheorem{isom.ext}[theorem]{{Trivial isometric extension}}
\newtheorem{lemma}[theorem]{{Lemma}}
\newtheorem{corollary}[theorem]{{Corollary}}
\newtheorem{fact}[theorem]{{Fact}}
\newtheorem{remark}[theorem]{{Remark}}
\newtheorem{remarks}[theorem]{{Remarks}}

\newtheorem{CONJ}[theorem]{{Projective Lichnerowicz conjecture}}
\newtheorem*{Kaehler}{{Theorem ({\bf Rigidity of $h$-projective transformation groups)}}}

\def\C{\mathbb{C}}
\def\R{\mathbb{R}}

\def\Z{\mathbb{Z}}

\def\la{\lambda}

\def\Iso{\sf{Iso}}
\def\Pro{\sf{Proj}}
\def\Aff{\sf{Aff}}

 \def\Jac{\sf{Jac}\;}

\def\L{\mathcal{L}}

\def\SS{\mathcal S}

\def\Diff{{\sf{Diff}}}

\def\PGL{{\sf{PGL}}}

\def\GL{{\sf{GL}}}
\def\SO{{\sf{SO}}}

\def\SU{{\sf{SU}}}

\def\SL{{\sf{SL}}}

\def\End{{\sf{End}}}

\def\Mat{{\sf{Mat}}}

\def\Aff{{\sf{Aff}}}
\def\SU{{\sf{SU}}}

\def\det{{\sf{det}}}

\addtocounter{section}{0}             
\numberwithin{equation}{section}       

\begin{document}

\setlength{\baselineskip}{0.53cm}   
%
%
\title[Projective group]
{On  discrete       projective transformation groups  of Riemannian manifolds}
\date{ \today}
\author{Abdelghani Zeghib \\
${}$ \\
 } 
\address{CNRS \\
UMPA \\
\'Ecole
Normale Sup\'erieure de Lyon\\
France}
\email{abdelghani.zeghib@ens-lyon.fr
\hfill\break\indent
 \url{http://www.umpa.ens-lyon.fr/~zeghib/}}
 \thanks{Partially supported by the ANR project GR-Analysis-Geometry}

\begin{abstract} 

 We prove rigidity facts for groups acting on pseudo-Riemannian manifolds  by preserving 
 unparameterized geodesics.

\vspace{0.35cm}

\noindent{\sc{R\'esum\'e.}} Nous d\'emontrons des r\'esultats de rigidit\'e pour les  groupes
agissant sur des vari\'et\'es pseudo-riemanniennes en pr\'eservant leurs g\'eod\'esiques non-param\'etr\'ees. 

\end{abstract}

\maketitle

\setcounter{tocdepth}{1}

\section{Introduction}

\subsubsection{The projective group of a connection} Two linear connections $\nabla$ and $\nabla^\prime$ on a manifold $M$ are 
equal iff they have the same (parameterized) geodesics. They  are called projectively equivalent if they have the same unparameterized geodesics. This is equivalent to that the difference $(2,1)$-tensor $T = \nabla - \nabla^\prime$ being trace free in a natural sense \cite{Eastwood}.

The affine group $\Aff(M, \nabla)$ is that of transformations preserving $\nabla$ and the projective one $\Pro(M, \nabla)$ is that of transformations $f$ sending $\nabla$ to a
projectively equivalent  one. So, elements of $\Aff$ are those preserving (parameterized) geodesics and those of $\Pro$ preserve 
unparameterized geodesics.

Obviously $\Aff \subset \Pro$; and it is natural to look for special connections for which this inclusion is proper, that is when projective non-affine transformations exist?

\subsubsection{Case of Levi-Civita connections} Let now $g $ be a Riemannian metric on $M$ and $\nabla$ its Levi-Civita connection. The affine 
and projective groups $\Aff(M, g)$ and $\Pro(M, g)$ are those associated to $\nabla$. 

More generally, $g$ and $g^\prime$ are projectively equivalent if so is the case for their associated Levi-Civita connections. This defines an equivalence relation on the space $\sf{Riem}(M)$ of Riemannian metrics on $M$.  Let $\mathcal P(M, g)$ denote the class of $g$, i.e. the set of metrics shearing the same unparameterized geodesics with $g$. It contains $\R^+. g$, the set of constant multiples of $g$. Generically, $\mathcal P(M, g) = \R^+.g$.  

One crucial fact here is that $\mathcal  P(M, g)$ is always a finite dimensional manifold whose dimension is called the {\it degree of projective mobility} of $g$.
  (This contrasts with the case of projective equivalence classes of connections which are infinitely dimensional affine spaces. Similarly, conformal classes of metrics are identified to spaces of positive functions  on the manifold). It is actually one culminate fact of projective differential geometry to identify   $\mathcal P(M, g)$ to an open 
  subset of a finite dimensional  linear sub-space  $\L(M, g)$ of endomorphisms of $TM$ (see  \S \ref{linear}). Being projectively equivalent for connections is a linear condition, but this is no longer linear for metrics (say because the correspondence $g  \to$ its Levi-Civita connection, is far from being linear!).  The trick is to perform a transform leading to a linear equation, see \cite{Bol-Mat} for a nice exposition.

\subsubsection{Philosophy} The idea  behind our approach here is to let a diffeomorphism
$f$ on a  differentiable manifold $M$ act on the space $\sf{Riem}(M)$ of Riemannian metrics 
on $M$. That this   action has a fixed point means exactly that  $f$ is an isometry for some 
Riemannian metric on $M$. One then naturally wonder what is the counterpart of
the fact that the $f$-action preserves some (finite dimensional) manifold $V \subset \sf{Riem}(M)$. 
A classical similar idea is to let the  isotopy class of a diffeomorphism on a surface act on  its Teichmuller space \cite{Thurston}.
Here, as it will be seen bellow,  we are 
  specially concerned with the case $\dim V = 2$.

\subsubsection{More general pseudo-Riemannian framework} All this generalizes to the  pseudo-Riemannian case. One fashion to unify all  
is to generalize all this  to the wider framework  of  second order ordinary differential equations (e.g.   hamiltonian systems) on $M$, by letting their solutions 
playing the role of (parameterized) geodesics.

\subsubsection{Rigidity  of the projective group} We are interested here in a (very) natural and classical problem
in differential geometry: {\it Characterize pseudo-Riemannian manifolds $(M,g)$ for which 
$\Pro (M, g) \supsetneqq \Aff (M, g)$, that is $M$ admits an essential projective transformation?  }  Constructing  upon a long research history by many people (see for instance \cite{Mat1, Mat2}), 
  we dare formulate more precisely:

\begin{CONJ} Let $(M, g)$ be a compact 
pseudo-Riemannian manifold. Then, unless $(M, g)$ is a finite quotient of the standard
Riemannian sphere,  \\
$\Pro(M, g)/ \Aff(M, g)$ is finite. 

--- Same question when compactness is replaced by completeness (this does not contain the first case since general   pseudo-Riemannian non-Riemannian compact manifolds 
may be non-complete).

\end{CONJ}

\subsubsection{Gromov's vague conjecture} It states that rigid geometric structures having a large automorphism group, are classifiable \cite{Gro1, Gro2}!  This needs a precise experimental (realistic) formulation for each geometric structure.  Our question above is an optimistic 
 formulation in the case of metric projective connections (those which are of Levi-Civita type). The historical case was that of Riemannian conformal 
structures with a precise formulation, as in the projective case above with the sphere playing a central role, is generally attributed to Lichnerowicz, and solved by J. Ferrand \cite{Ferrand,  Obata}. In the general conformal  pseudo-Riemannian case,  there are many ``Einstein universes'', i.e.  conformally  flat examples with an essential conformal group. 
A Lichnerowicz type conjecture  would  be that all pseudo-Riemannian manifolds with an  essential conformal group are conformally flat. 
 However,  this was recently invalidated  by C. Frances \cite{Frances}. In the projective case, 
there is no natural candidate of a compact pseudo-Riemannian  (non-Riemannian) manifold 
playing the role Einstein universes; 
it becomes a natural 
challenge to prove that indeed $\Pro/ \Aff$ is always finite in this situation?

 In the vein of this vague conjecture, it is surely interesting  to see to automorphism groups of non-metric projective connections...

\subsection{Results} This   very classical subject of differential geometry was  specially investigated by the Italian and next the Soviet  schools. 
 All  famous names: Beltrami, Dini, Fubini, Levi-Civita  are still involved in results  on  projective equivalence of metrics \cite{Beltrami, Dini, Levi}. As for the ``Soviet'' side,  let us quote o \cite{Aminova1, Aminova2, Sin, Mat1, Mat2, Solo1, Solo2}, and as names 
 Solodovnikov who ``introduced'' the projective group problem, and last 
 V. Matveev, who handled  many remarkable cases of it.

\subsubsection{Killing fields variant} 
Actually, it was  $\Pro^0(M, g)$, 
 the identity component of $\Pro(M, g)$,  that got real  interest in the literature. Its elements are those 
 belonging to flows of projective Killing fields.  There is a prompt formulation of the Lichnerowicz  conjecture here: {\it if $\Pro^0(M, g) \supsetneqq \Aff^0 (M, g)$, then
 $(M, g)$ is covered by the standard sphere} (assuming $M$ compact). 
 
 This identity component variant was proved by V. Matveev in the case of Riemannian manifolds \cite{Mat1}, and remains   open in the case of higher signature.
 
 Local actions, i.e. projective Killing fields with flows defined only locally, were also considered, see for instance \cite{Bryant}.
 
  However, situations with no Killing fields involved, say for example when $\Pro$ is a discrete group do not seem to be studied. We think it is worthwhile to consider them because the discrete part may have dynamics stronger than the connected one, as in the case of a flat torus $\mathbb T^n$, but in fact for its affine group whose discrete part is the beautiful arithmetic (the best!)  group  $\SL_n(\Z)$.
 
 \subsubsection{Non dynamical variant} Without actions, one may think of having big $\mathcal P(M, g)$ as an index of symmetry, and one naturally may ask when this happens.  For this, as in the projective case, consider $\mathcal A(M, g)$,  the set 
 of metrics  affinely equivalent to $g$ (i.e. having the same Levi-Civita connection). 
   Here, we have the following wonderful  theorem:
 
 \begin{theorem} [Kiosak, Matveev, Mounoud, see \cite{Mat2}]  \label{higher.rank} Let $(M, g)$ be a compact pseudo-Riemannian manifold.  If $\dim \mathcal P(M, g) \geq 3$, 
 then $\mathcal P(M, g)$= $\mathcal A (M, g)$, unless $(M, g)$ is  covered by the standard Riemannian sphere. In particular
  $\Pro(M, g) = \Aff(M, g)$ in this case.
 
 \end{theorem}

\subsubsection{Rank 1 case?} In view of this, it remains to consider the case $\dim \mathcal P(M, g) = 2$ (the dimension 1 case is trivial).   Actually, this case occupies a large part in proofs of Lichnerowicz conjecture in the Riemannian 
as well as K\"ahlerian cases \cite{Mat1, Mat3, Mat4}. 
(We think our approach here, besides it treats the discrete case,  also  simplifies these existing proofs).  We are not surprized of 
the resistance of this  case, reminiscent to a rank 1 phenomena, vs the higher rank case. Assuming 
$\dim \mathcal P(M, g) \geq 3$ hides a symmetry abundance hypothesis!

Anyway, in all our proofs, we will assume $\dim \mathcal P(M, g) = 2$.

\subsubsection{Aim} Our first objective here is to  provide a proof of the above conjecture   in case 
of compact Riemannian manifolds

\begin{theorem} \label{Riemannian} Let $(M, g)$ be a compact Riemannian manifold.  If 
 $M$ is not a Riemannian finite quotient of a standard sphere, 
and $\Pro(M, g) \supsetneq Aff(M,g)$, then $\Pro(M, g) $ is a finite extension of $\Aff(M, g)$. 

More precisely, 
  $\Aff(M, g) = \Iso(M, g)$, and a subgroup $\Iso^\prime(M, g)$ of index $\leq 2$, is normal 
  in $\Pro(M, g)$, and the quotient group 
 $\Pro(M, g) / \Iso^\prime(M, g)$ is  either  cyclic of order  $\leq  \dim M$, or dihedral of order $\leq 2 \dim M$.

\end{theorem}

\subsection*{Examples} In order to illustrate the non-linear character of projective equivalence, let us recall the Dini's 
classical 
result: two metrics on a surface are projectively equivalent, iff, at a generic point, they have the following forms in some coordinate system:
   $$g= (X(x) -Y(y) (dx^2 + dy^2), \; \; \bar{g} = (\frac{1}{Y(y)} - \frac{1}{X(x)})(\frac{dx^2}{X(x)} + \frac{dy^2}{Y(y)})$$
It follows that for   the metric $g = (a(x) - \frac{1}{a(y)}) ( \sqrt{ a(x)} dx^2 + \frac{1}{\sqrt{a(y)}} dy^2)$, the involution $(x, y) \to (y, x)$ is projective.  This example given by V. Matveev \cite{Mat1},  shows that non-affine projective transformations may exist (outside
the case of spheres) but are not in the identity component, because of  his result. Theorem \ref{Riemannian} says that
the ``discrete projective transformation group'' is always finite, but we do not know
examples more complicated than the last involution.

\begin{remark} Some of quoted results are also true in the complete non-compact case, but we 
consider here compact manifolds, only.
 
\end{remark}

\subsection{K\"ahler version}   \label{Kaehler}  Let $(M, g)$ be a Hermitian manifold. Let $V \subset M$ be a geodesic surface which is at  the same time
a holomorphic curve.   If $g$ is K\"ahler, then  any (real) curve $c$ in $V$ satisfies that its complexified tangent direction is parallel; it is  therefore 
called {\it $h$-planar}.  It is very special that such $V$ exists, but $h$-planer curves   always exist. 
Two  K\"ahler  metrics are $h$-projectively equivalent if they share the same $h$-planer curves. A holomorphic diffeomorphism  $f$ is {\it $h$-projective} if $f_*g$ is $h$-projectively equivalent to $g$. Their group is denoted $\Pro^{Hol}(M, g)$. (There exist equivalent  terminologies for $h$-projective, as holomorphic-projective or $c$-projective).

This holomorphic side of the projective transformation problem 
was classically investigated  by the Japanese school \cite{Has, Ishi, Yano, Yoshi}. 

Finally, 
V.  Matveev and  S.  Rosemann  generalize all known Riemannian results (on the identity component) to   the  K\"ahler  case \cite{Mat3}.  That is, 
if  $\Pro^{Hol}(M, g)$  
contains a one parameter group of non-affine transformations, then, up to a scaling, 
$(M, g)$  is holomorphically isometric to  $\mathbb P^d (\C)$ endowed with its standard metric (where 
$d = \dim_\C M$).

Like in the Riemannian case, we are able here to handle the discrete part of  $\Pro^{Hol}$:

\begin{Kaehler} Let $(M^d, g)$ be a  compact K\" ahler manifold. If 
 $\Aff^{Hol}(M, g)$  has not finite index in  $\Pro^{Hol}(M, g)$, then, 
 up to 
 a scaling,
$(M, g)$ is holomorphically isometric to $\mathbb P^d(\C)$ endowed with its Fubini-Study metric $g_{SF}$.
\end{Kaehler}

\subsubsection*{About the proof} 
V.  Matveev and  S.  Rosemann proved their K\" ahler  identity component version by showing that  
all the differential geometric tools developed in the (usual Riemannian)  projective case, may be adapted to the $h$-projective one, and 
  enjoy all the needed   properties, see \cite{Mat3} for details. 
  Thanks  to this, we will not give details of proof in the $h$-projective case, because it goes exactly as  in the (usual) projective one. Instead, we investigate the following new aspects in the K\" ahler case, in particular to in order to elucidate another use of the word "projective''!

\subsubsection*{Projective vs projective} Recall that a complex manifold $M$ is called projective if it is holomorphic to 
a (closed regular) complex submanifold of some projective space $\mathbb P^N\C)$. Endowed with the restriction of 
 $g_{SF}$,  $(M, {g_{SF}}_{\vert M})  $ is a K\"ahler manifold. However, only few (other) K\"ahler metrics $(M, g)$ admit holomorphic isometric embedding in  a projective space (but, of course real analytic isometric embedding exist, by Nash Theorem).  The dramatic example is that of an elliptic curve, that is a 2-torus with a complex structure. It admits a large  space of holomorphic embedding in projective spaces of different  dimensions, but the induced metric on them can never be flat!  This is one case of a ``Theorema Egregium'' due to Calabi \cite{Calabi} which says that holomorphic isometric immersions in space forms of  constant holomorphic sectional curvature, are absolutely  rigid (see \S \ref{K\"ahler.case}).

\begin{theorem}  \label{K\"ahler} Let $(M^d, g)$ be a  complex submanifold of a projective space $\mathbb {P}^N(\C)$ endowed with the induced metric
(from the normalized Fubini-Study).  Then the group $\Pro^{Hol}(M, g)$ of holomorphic projective transformations  is a finite extension 
of $\Iso^{Hol}(M, g)$, its  group of holomorphic isometries, unless $(M, g)$ is holomorphically homothetic  to $\mathbb P^d(\C)$.  More precisely,  up to composition with $\SU(1+N)$, 
$M $ is the image of a Veronese map: $v_k: \mathbb P^d(\C) \to \mathbb P^N(\C) $ (which expands  the metric by a factor $k$).

\end{theorem}

\begin{remark} There are submanifolds $M \subset \mathbb P^N(\C)$ with a big ``projective'' group, say such    that 
 $G_M = \{g \in \GL_{N+1}(\C), g.M = M\}$ is non-compact and acts transitively on $M$. So, $G_M$
 does not  act  projectively with respect to the induced metric, unless $M$ is a Veronese submanifold. The $G_M$-action preserves another kind of geometric structures? It is however remarkable that all the automorphism group of any K\"ahler manifold preserves
 a huge class of minimal submanifolds (in the sense of Riemannian geometry), namely, complex submanifolds!

\end{remark}

\subsection{Towards the general (indefinite) pseudo-Riemannian case}

It was proved in \cite{Mat2}
that the quotient space $\Pro^0(M, g) /Aff^0(M, g)$ has always  dimension $ \leq 1$.  We have the following generalization to  full groups.

\begin{theorem}  \label{pseudo}  Let $(M, g) $  be a compact pseudo-Riemannian manifold  having an essential projective group, that is,  $\Pro(M,g)/Aff(M,g)$ is infinite. Then, up to finite index:

1)   $\Aff(M, g) = \Iso(M,g)$  and it is  a  normal subgroup of $\Pro(M, g)$. 

2) $\Pro(M, g)/\Iso(M, g)$  is  isomorphic to  a subgroup of  $\R$. More precisely,  
 there is a representation $ \Pro(M, g) \to 
\SL_2(\R)$ whose kernel is $\Aff(M, g)$ and range   contained in a non-elliptic  1-parameter  group.

\end{theorem}

\subsubsection{Organization} We restrict ourselves here to compact manifolds, and from \S \ref{homography} to
the case of metrics of projective mobility $\dim \mathcal P(M, g) = 2$.

Our proofs are  mostly algebraic, somewhere dynamical but rarely geometrical!

 \subsubsection*{Acknowledgements}
 I would like to thank L. Florit,  	 A. J.  Di Scala, V. Matveev and P.  Mounoud for their help, and especially the referee for many interesting 
 remarks and suggestions.

 \subsubsection{Added to last version} Around one year and half  after the  publication of
 our present work (in ArXiv), V. Matveev   \cite{Mat6} improved  estimate in Theorem 
 \ref{Riemannian}
 of  the (finite) index of 
 $\Aff(M, g)$ in $\Pro(M, g)$.  He shows that this index is $\leq 2$ (of course when
 $(M, g)$ is not homothetic to a quotient  of the standard sphere). Matveev's proof consists 
 in pursuing analysis of  our elliptic case \S \ref{all.elliptic}  (otherwise, he bases on our results here, especially 
 in the hyperbolic case). 

\section{Actions, General considerations}

\label{general}

$M$ is here a compact smooth manifold.

\subsubsection{}  Let $\mathcal E$ be the space  of $(1,1)$-tensors $T$, i.e. sections of the linear bundle $End(TM) \to M$: for any $x$, $T_x$ is linear map $T_xM \to T_xM$.  The space 
 $\mathcal E$    has a natural structure of algebra  with unit element $I$ the identity of $TM$ (over $\R$ as well as over $C^\infty(M)$ or $C^k(M)$).

$\Diff(M)$ acts  naturally on $\mathcal E$ by $(f, T) \to \rho^E(f)T= f_*T$ defined naturally 
 by $(f_*P)_x= D_{f^{-1}x}fT_{f^{-1}x}D_xf^{-1}$.

\subsubsection{}  Let $\mathcal G$ be the space of pseudo-Riemannian metrics on $M$.  Then,  $\Diff(M) $ acts on $\mathcal G$ by   taking direct image, $(f, g) \to  \rho^G(f)g=  f_*g$  defined by
$(f_*g)_x(u, v) = g_{f^{-1}(x)}((D_xf)^{-1} u, (D_xf)^{-1} v)$.

\subsubsection{Notation} We will sometimes use the usual notations $f_*T$ and $f_*g$ for $\rho^E(f)T$ and $\rho^G(f)g$, respectively. 

\subsubsection{Transfer}
Given a metric  $g_0$ on $M$, any other metric ${g}$ 
can by written ${g}(., .) = g_0(T., .)$, where  the {\ transfer} tensor $T = T({g}, g_0)$
is a  $g_0$-symmetric$ (1, 1)$-tensor  
(i.e. $T_x$ is   a symmetric endomorphism of $(T_xM, {{g}_0}_x)$).  

In fact, a metric $g$ defines a bundle isomorphism $TM \to T^*M$, and thus  $T(g,g_0) = g_0^{-1}g$.

In other words, we have a map $C_{g_0}:  g \in \mathcal G \to T(g,g_0) = g_0^{-1}g  \in \mathcal E$. In particular $C_{g_0} (g_0) = I$ (the identity of $TM$).

\subsubsection{Transfer action} The transfer of the natural $\Diff$-action $\rho^E$  on $\mathcal G$ to $\mathcal E$  by means of $C_{g_0}$, is by definition 
$$(f, T) \to 
\rho^{GE}(f)(T)= C_{g_0} (\rho^G(f) (C_{g_0}^{-1}T))$$
 It equals:
$${g_0}^{-1}  \rho^G(f) (g_0 T) = 
g_0^{-1}( \rho^G(f) g_0) (\rho^E(f)T)  =  
  S_f  \rho^E(f)T$$ 
  where the $g_0$-strength of $f$  is $S_f = 
  g_0^{-1} ( \rho^G(f) g_0) $.

  \subsubsection{A preserved functional} 
  The following  ``norm-like'' functional $Q(T)= \int_M \sqrt{ \vert \det T \vert} dv_{g_0}$  is preserved by $\rho^{GE}$.
   Indeed,  $$Q(\rho^{GE}(f)T)= \int_M \sqrt{ \vert (\det  S_f) \det (f_* T)  \vert} dv_{g_0}  =   \int_M \sqrt{  \vert \det T  \vert}  (f^{-1}(x)) \Jac_x f^{-1} d v_{g_0}, $$
   and this equals $Q(T)$.

\subsubsection{} Consider now the  partially defined transform  $\mathcal F: L \in \mathcal E \to 
  T = \frac{L^{-1}}{\det L}  \in \mathcal E$.    Its inverse map is given by ${\mathcal F}^{-1}(T) = (\det T)^{\frac{1}{1+d}} T^{-1}$ ($d = \dim M$).  
  
  It is remarkable that $\mathcal F$ commutes with the $\Diff$-action $\rho^E$ on  $\mathcal E$.  The finite dimensional version of this for a linear space $E$   is that $u  \to \End(E) \to \frac{u^{-1}}{\det u}  \in \End(E)$ commutes with the $\GL(E)$ action given by  $(A, u) \in \GL(E) \times \End(E) \to   AuA^{-1} \in \End(E)$

  \subsubsection{Action in the $L$-representation} Consider now the map $$g \in \mathcal G \to  L = \mathcal  F^{-1} \mathcal (C_{g_0} (g)) \in \mathcal E$$ In other words, to a metric $g$, we associate the $(1,1)$-tensor 
  $L$ such that   $g(., .) = \frac{1}{\det L}g_0(L^{-1}., .)$. 
  
  The corresponding action $\rho$ on $\mathcal E$ is given by: $$ \rho(f) L = (\rho^E(f)L) K_f $$ where $K_f$, the $g_0$-strength of $f$ in the $L$-representation, is  the $\mathcal F^{-1}$-transform of $S_f$, that is $K_f$ is defined by 
  $\rho^G(f)g_0(., .) = \frac{1}{\det K_f} g_0(K_f^{-1}., .)$.
  
Corresponding to $Q$,   $\rho$ preserves the partially defined  functional: $L \to  N(L) = \int_M \frac{1}{ \det L^{(1+d)/2}} dv_{g_0}$

\subsubsection{The chain rule for strength}

   $$K_{f^n}=  ({f_*^{n-1}}K_f) ({f_*^{n-2}}K_f) \ldots (f_*K_f) K_f$$ (of course $(f^k)_* = (f_*)^k$).

 \subsubsection{} Summarizing:
 \begin{fact} \label{summer}  Let $g_0$ be a fixed metric on $M$.  To any metric $g$, let $L$ be the $(1,1)$ tensor defined by $g(., .) = \frac{1}{\det L}g_0(L^{-1}., .)$.  The $\Diff$-action on $(1,1)$tensors deduced from the usual action on metrics by means of this map $g \to L$ is given by 
 $$(f, L) \in \Diff(M) \times \mathcal E \to \rho(f)L= (f_*L )K_f$$
  Here $K_f$ is the $L$-tensor associated to $f_*g_0$, i.e. $f_*g_0 = \frac{1}{\det K_f} g_0(K_f^{-1}., .)$, and 
 $f_*L$ denotes the usual action on $\mathcal E$.

-  $f$ is an isometry of $g_0$ $\iff $  $K_f = I$

 -  $f$ is a $g_0$-similarity (that is $f_*g_0= b g_0$ for some constant $b$)  $\iff$ $K_f = bI$ for some $b$.
 
 -   $\rho$  preserves the function $L \to N(L) = \int_M \frac{1}{ \det L^{(1+d)/2}} dv_{g_0}$
 
 \end{fact}

\section{Linearization, Representation of $\Pro(M, g)$ in $\L(M, g)$}

\label{linear}
(We will henceforth mostly  deal with only one metric and so we will denote it $g$ instead of $g_0$).

\subsubsection{The space $\L(M,  g)$} 
Recall that $\mathcal P(M, g)$
denotes the 
class  of metrics projectively equivalent 
to $g$. 
 
 Let $\L^*(M, g)$  
 be  the image of $\mathcal P(M, g)$ under the correspondence of Fact \ref{summer},  and $\L(M, g) $
 its linear hull:
   $$\L(M, g) =     \{ L = \Sigma_i a_i L_i,  a_i \in \R,  \;  \hbox{ such that}  \; \frac{1}{\det L_i}g(L_i^{-1}., .)
\; \hbox{is projectively equivalent to}  \; g \}$$

Let us call $\L$-tensors the elements of this space.

\subsubsection{Linearization}

\begin{theorem}  \cite{Bol-Mat, Sin} \label{Bol-Mat, Sin} 
$L \in \L(M, g)$ iff $L$ satisfies the linear equation:
$$g((\nabla_u L)v, w) = \frac{1}{2}g(v, u) d \sf{trace}(L)(w) + \frac{1}{2}g(w, u) d \sf{trace}(L)(v) $$
where $\nabla$ is the Levi-Civita connection of $g$.

Furthermore:

 - $\L^*(M,  g) $ is an open subset of $\L(M, g)$: an element $L \in \L(M,  g)$ belongs to $\L^*(M, g)$ iff it is an isomorphism of 
 $TM$.
 
- $\L(M, g)$ has finite dimension (bounded by that corresponding to the projective space of same dimension).

- $L \in \L^*(M, g)$ is parallel iff the corresponding metric $ \frac{1}{\det L}g(L^{-1}., .)$ is affinely equivalent to $g$, iff
$L$ has constant eigenvalues. 
\end{theorem}

\subsubsection{Linear representation of $\Pro(M, g)$}

\begin{fact} \label{injective}

We have a finite dimensional representation,
$$f \in \Pro(M, g) \to  \rho(f)  \in \GL( \L(M, g))$$
where $\rho(f)(L) = f_*(L).K_f$.

 $\bullet$ $\rho$  preserves the norm-like 
function $N(L)  = \int_M \frac{1}{ \det L^{(1+d)/2}} dv_{g}$.

$\bullet$ Let 
$p: \GL (\L(M, g)) \to \PGL(\L(M, g))$ be the canonical projection , then $p$ is injective on $\rho (\Pro(M, g))$, 
or has at most a  kernel $ \cong \Z/2 \Z$.

$\bullet$ Let $\mathcal D$ be the subset of degenerate  tensors in $\L(M, g)$: $$\mathcal D = \{ L \in \L(M, g), \; L  \hbox{ not an isomorphism of } \; TM \}$$
Then $\mathcal D$ is a closed cone invariant under $\rho$.

\end{fact}

\begin{proof}  

${}$

- The first point is imported from Fact \ref{summer}

-   For the second one,  let  $a A$ and $A$ in $\GL( \L(M, g))$ such that  both preserve $N$, then $N(a A(L)) = N(A(L)) = N(L)$, for any $L$.  But $N(a L) = \vert a \vert ^s N(L)$ with $s = - d(d+1)/2$, and hence $a = \pm 1$. 

- To prove $\rho$-invariance of $\mathcal D$, observe that $L \in \mathcal D$   iff  for some $x \in M$, $\det L (x) = 0$. But $\rho(f) L = f_* L  K_f$, and hence $\det (\rho(f) L) (f(x))  = \det L (x) \det K_f(x) = 0$

\end{proof}

\begin{remark} Actually  $\mathcal D$  coincides essentially with the $\infty$-level of $N$.

\end{remark}

\section{The Case $\dim \L(M, g) = 2$,  a Homography}

\label{homography}

\subsubsection{Hypothesis} Henceforth, we will assume 
  that   $\dim \L(M, g)= 2$.

 Fix    $f$ that is   not homothetic, i.e.  $K= K_f$ is
not a multiple of $I$. Hence $\L(M, g)$ is 
  spanned by $K$ and $I$.

\subsection{The degenerate set $\mathcal D$}

 \begin{fact} 
 \label{sectors}
The subset of degenerate tensors (defined above) satisfies:  $$ \mathcal D = \{ a(K -t I),  a \in \R, \: \hbox{and} \;   t  \;  \hbox{a real spectral value of } K : \det(K(x)-tI) = 0\;\;  \hbox{for some }\; x\}$$ 
 
 In particular $I$ and $K_f$ $\notin \mathcal D$. 
 
If the spectrum is real and described by $d$ continuous eigenfunctions 
   $x \to \la_1(x) \leq   \ldots \leq \la_d(x)$ ($d = \dim M$), then  $ \mathcal D = \cup_{i= 1}^{i= d} (C_i \cup -C_i)$, where 
   $$C_i = \{ a(K -t I),  a \in \R^+, \: \hbox{and} \;   \inf \la_i \leq t  \leq \sup \la_i \}$$ 
  
  Each $C_i$ is a proper convex  cone (sector). 
  
  Finally, unifying intersecting sectors, we get a minimal  union:  $ \mathcal D = \cup_{i= 1}^{i= k} (D_i \cup -D_i)$, where the $D_i$
  are disjoint sectors.

 \end{fact}

 \begin{proof} $I $ (as well as $K$) do not belong to $\mathcal D$ and hence any element of this set has the form $a(K - tI)$. This belongs
 to $\mathcal D$ iff $\det(K(x) -tI) = 0$, for some $x$, that is $t \in \cup_i (Image (\la_i))$, and the cones $C_i$ follow.

 \end{proof}

    \subsection{Action by  homography} 
 \subsubsection{Equation}
 By  the 2-dimensional assumption, there exist $\alpha, \beta$
 such that:
 $$\rho(f)(K)= (f_*K) K = \alpha K + \beta I$$
  Equivalently, 
 \begin{equation*} \label{action}
  f_*K = \alpha  I + \beta K^{-1} 
 \end{equation*}
  Say somehow 
 formally,  $f_*K = \frac{\alpha K + \beta I}{K}$.

  Since $f_*I = I$, in the basis $\{K, I\}$,  $\rho(f): L \to f_*(L).K_f$
  has a matrix 
 $$B=  B_f= \begin{pmatrix}\alpha & 1 \\
 \beta & 0 
 \end{pmatrix}$$

\subsubsection{}
The group  
 $\GL_2(\R)$  (more faithfully $\PGL_2(\R)$) acts 
on the (projective) circle 
 $\mathbb S^1 = \bar{\R} = \R \cup \infty$, by means of the law
 $$z \to   A\centerdot  z=    \frac{a z + b }{cz+d}\;\;  \hbox{for} \;  A =  \begin{pmatrix} a &  b \\
 c & d 
 \end{pmatrix}\in \GL_2(\R)$$

In fact, we can also let $\GL_2(\R)$ act on the 
space of $(1,1)$-tensors by the same formula: $(A\centerdot X)(x)= (aX(x) + bI)(cX(x) +dI)^{-1}$.
In other words, the action is   fiberwise, and when a fiber 
$\sf{End}(T_xM)$ is identified to $\Mat_n(\R)$, then  $A \centerdot X =     \frac{a X + b }{cX+d}$

Now, the previous equation $f_*K = \frac{\alpha K + \beta I}{K}$ can be interpreted by that the linear $f_*$-action on $K$ equals 
the homographic action $A \centerdot K$ where 
 $\label{matrix}
 A=   \begin{pmatrix}\alpha &  \beta \\
 1 & 0 
 \end{pmatrix}$ is the transpose of $B$.

\subsubsection{Iteration}
We have:

\begin{equation}
  {f_*}^nK= A^n\centerdot  K,  \;  \hbox{for any }\; n \in \Z
  \end{equation}

This can be proved in a formal way. Let 
  $C$ be an endomorphism on an abstract algebra   $\{1, x, x^{-1}, \ldots\}$,
  such that $C(x) = \alpha + \beta x^{-1}$ (with 
  $C(1) = 1$ and $C(x^{-1}) = C(x)^{-1}$). Then, $C^n(x)= A^n \centerdot  x $,  where 
  $A = \begin{pmatrix}\alpha &  \beta \\
 1 & 0 
 \end{pmatrix}$.

\subsubsection{Significance for eigen-functions} \label{equivariance}

Let $x \in M $ and
$y = f(x)$ and denote $T= D_xf: T_x M \to T_yM$.

 The relation  
 $f_*K= \alpha I + \beta K^{-1}$  means that $  T^{-1} K_y T = \alpha  + \beta K_x^{-1}$. 
 This  implies  in particular that 
$T$ maps an invariant subspace of $(T_xM, K_x)$ to an invariant 
subspace of $(T_yM, K_y)$. If $E_\lambda (x) \subset T_xM$ is the (generalized) $K_x$-eigenspace associated to   $\la$, then 
$T$ maps it to $E_{ A^{-1} \centerdot \la}(y)\subset T_y M$. 

Let $x \to Sp(x) \subset \C$  be the multivalued spectrum function of $K$, that is   $Sp(x) \subset \C$ is the set of eigenvalues of $K_x$. Then the image $A \centerdot Sp(f(x))$
(of the subset $Sp(f(x))$  under the homography $A^{-1}\centerdot$) equals
$Sp(x)$, and so $$
Sp(f(x)) = A^{-1}\centerdot Sp(x)$$

Also, if $\la: M \to \C$ is a continuous $K$-eigen-function, that is $\la(x) \in Sp(x)$ for any $x \in M$ and $\la$ is continuous,
 then $$x \to  \la^\prime (x)= A^{-1} \centerdot \la(f^{-1}(x)) $$ is another continuous $K$-eigen-function.

\subsection{Classification of elements of $\SL_2(\R)$} \label{classification.elements} 
Recall   that non trivial elements $A$ of $\SL_2(\R)$ split into three classes:

\begin{enumerate}

\item  Elliptic: $A$  is  conjugate in $\SL_2(\R)$ to a rotation, i.e. an element of $\SO(2)$. Its homographic 
action on $\bar{\R} = \R \cup \infty$, as well as on $\bar{\C} = \C \cup \infty$ is conjugate to a rotation.

\item Parabolic: $A$ is unipotent, i.e. $(A-1)$ is nilpotent.  Its homographic action on $\bar{\R}$ as well as on $\bar{\C}$ is conjugate 
to a translation. It  has a unique fixed point $F_A \in \bar{\R}$. Up to conjugacy
$F_A = \infty$, and $A\centerdot z \to z +a$, where $a \in \R$.

It follows that if $C \subset \R$ is a bounded  $A\centerdot$-invariant set, then $C= \{F_A \}$ (and necessarily $F_A \neq \infty$).

\item Hyperbolic: A has two fixed points $F_A^-$ and $F_A^+$. Up to conjugacy, 
$F_A^- = 0$, $F_A^+ = \infty$, and $f(x) = a x$, with $0<a<1$. 
\end{enumerate}

Now, if $A \in \GL_2^+(\R)$, its homographic action coincides with that of $\frac{A}{\det  A}$, and the same classification applies.

\section{Riemannian metrics: non-hyperbolic cases}
\label{non.hyperbolic}

$(M, g)$ is  here a compact Riemannian manifold, and $f$ as in the previous section a chosen element such that 
$\{ K_f, I \}$ generate $\L(M, g)$, and $f$ is not affine. 

In this Riemannian setting,  all elements of 
$\L(M, g) $ are diagonalizable (since self-adjoint).

Let $G  = \rho(\Pro(M, g))$, and  $G^+ = G \cap \GL_2^+(\R)$.

Let $  A = \rho(h) \in G^+$  be non-trivial,  then $K_h$ is not collinear to $I$.
Indeed,  assume  by contradiction that $K_h = a I$, then recall that $h_*g(., .) = g(S_h., .)$ and  $S_h=  \frac{K_h^{-1}}{ \det K_h}$. 
It follows that $S_h$ has the form $S_h = bI$, but equality of volumes of $(M, g)$ and $(M,  h_*g)$ implies $b = 1$  
and hence also $K_h=  I$ (that is $h \in \Iso(M, g)$).  
Thus, 
$\rho(h) I = I$.  

On the other hand, by  Fact \ref{sectors}, $\rho(h)$ preserves a finite set of lines, all different from $\R I$.   Let  $l_1^+$ and $l_2^+$ be 
the two nearest 
half lines to $\R^+ I$.  If $\rho(h) \neq Id$, then necessarily $\rho(h) l_1^+ = l_2^+$, but this implies $\rho(h)$ is a reflection which contradicts our hypothesis $\det \rho(h) >0$. 

\subsubsection{$G^+$ can not contain parabolic elements.}  Assume by contradiction that  $\rho(h)$ is parabolic with fixed point $F_h$. Then, $F_h$ is the unique real spectral value of $K_h$ (because there is no other bounded set of $\R$ invariant under the associated homography),  and thus $K_h$ is proportional to $I$ (since it is diagonalizable), which we have just proved to be impossible.
 
\subsubsection{Case where all elements of $G^+$ are elliptic}  \label{all.elliptic}

Recall that we have a union of $k \leq \dim M$  disjoints sectors $D_i$, such that 
   $ \mathcal D = \cup_{i= 1}^{i= k} (D_i \cup -D_i)$ is $G$-invariant (Fact \ref{sectors}). If $k >1$, then the stabilizer of $\mathcal D$ in $\SL_2(\R)$ is compact, and we can assume 
   $G$ is a subgroup of $\sf{O}(2)$.  
   
   $G^+$ ($= G \cap \GL_2(\R)$) is a finite subgroup of $\SO(2)$ and hence cyclic.  
   Now, if a rotation   preserves  a set of $k$-disjoint sectors, then it has order $\leq 2k$. We will observe in our case that this order is in fact 
   $\leq k$. 
   
   However, in our case, we know that any  $\rho(h)  \in G^+$ is $\neq - Id$ (since otherwise $-I= \rho(h) I = h_*(I) K_h = K_h$, and hence $K_h = -I$ which we have already excluded).
   Say, in other words we can see the rotation acting on the projective space rather than the circle and get exactly $k$-sectors and deduce that  actually, $G^+$ has order $\leq k$. 
   
As for  $G$ (if strictly bigger than $G^+$), it is   dihedral of order $\leq 2 k$.
   
   Finally,  in the case $k = 1$,    that is $\mathcal D = D_1 \cup -D_1$, its  stabilizer in  $\SL_2(\R)$ contains $-Id$ together with 
   a one parameter hyperbolic group. So, if we assume all elements of $G^+$ elliptic, we get $G^+ =  \{1\}$.  In this case, $G$ itself reduces to a single reflection (if non-trivial).

   \subsubsection{About  $\Iso(M, g)$} Observe first that if $h \in \Aff(M, g)$, then necessarily $K_h$ is proportional to $I$ since 
   otherwise $K_h$ will be a  combination with constant coefficients of $I$ and
   $K_h$, and thus has constant eigenvalues, and therefore $f \in \Aff(M, g)$ (see  \ref{Bol-Mat, Sin}) contradicting our hypothesis. As  observed previously  $K_h = I$, 
   that is $h \in \Iso(M, g)$ and so $\Aff(M, g) = \Iso(M, g)$.
   
   On the other hand, if $h \in  \Iso(M, g)$, and $\rho(h) \in G^+$, then $\rho(h) = Id$.  
   
   In general,  if $\rho(h)  \neq Id$, then 
    it is a reflection since $\rho(h^2) = Id$. 
   
   Let $\Iso^{(2)}(M, g)$ be the normal subgroup 
   of $\Iso(M, g)$ generated by squares $h^2, h \in \Iso(M, g)$.  Then:
   
   - either  $\ker \rho = \Iso(M, g)$, 
   
   - or $\ker \rho = \Iso^{(2)}(M, g)$, and this has index 2 in $\Iso(M, g)$.
   
 - in all cases, $\Iso(M, g)$ or $\Iso^{(2)}(M, g)$ is normal in $\Pro(M, g)$, and the corresponding  quotient 
 is cyclic of order $\leq  \dim M$, or dihedral of order $\leq 2 \dim M$.

\section{Riemannian metrics, Hyperbolic case}
\label{hyperbolic}

In the present section,  
  $(M, g)$ is  a compact Riemannian manifold
with $\dim \L(M,g)= 2$ and $f \in \Pro(M, g)$   is such that $\rho(f)$ is hyperbolic. 

The final goal (of the section) is to prove that
$(M, g)$ is projectively flat. This will be done by proving the vanishing of its Weyl projective tensor $W$
(recalled below). For this, one  iterates a vector $z = W(u, v)w$ by the $Df$-dynamics to get 
a sequence $z_n = Df^nz = W (Df^n u, Df^n v)Df^n w$, and shows that it has two different growth rates
(when $n \to \pm \infty$) unless $z = 0$.

\subsection{Size of the spectrum}

\begin{fact} The homography $A \centerdot$ defined by $\rho(f)$ has two real finite fixed points 
$\la_-  < \la_+$. 

 $K = K_f$ has exactly one non-constant eigen-function $\la$. It has  multiplicity   1 
(at generic points), 
range the interval  $[\la_-, \la_+]$, and satisfies the equivariance: $\la(f(x)) = A^{-1}\centerdot \la(x)$.

The full spectrum of $K$ may be $\{\lambda_-, \lambda, \lambda_+ \}$, 
$\{\lambda_-, \lambda\}$ or $\{\lambda, \lambda_+ \}$. We denote the multiplicities
of $\lambda_-$ and $\lambda_+$ by $d_-$ and $d_+$ respectively, and hence $\dim M = 1 + d_1 +d_+$.

\end{fact}

\begin{proof} Let $\mu_1 (x)    \leq \ldots  \leq \mu_d(x) $ be the eigenfunctions (with multiplicity) of $K(x)$. From \ref{equivariance}, the 
map $\mu_i^\prime: x \to A^{-1} \centerdot \mu_i(f^{-1}(x)) $ is another eigenfunction and hence equals some 
$\mu_j$. Taking a power of $f$, we can assume $\mu_i^\prime = \mu_i$, that is $\mu_i(f^{-1}(x)) = A \centerdot \mu_i(x)$. In other words, 
$\mu_i$ is an equivariant map between the two systems $(M, f)$ and $(\R,  A^{-1}\centerdot)$. Thus, $Image (\mu_i)$ is a bounded
$A^{-1}\centerdot$-invariant interval. Hence $\la_\pm$ belong to $\R$ (rather than $\bar{\R}$) and the image of 
$\mu_i$ can be $\{ \la_-\}$, $\{ \la_+\}$ or $[\la_-, \la_+ ]$.  The fact that only one $\mu_i$ has range 
$[\la_-, \la_+]$ follows from the following  nice fact:
 :

\end{proof}

\begin{theorem} \cite{Mat5} Let $(M, g)$ be a complete Riemannian manifold and $L \in \L(M, g_0)$. 
Then two eigen-functions $\mu_i \leq \mu_j$ satisfy $\sup \mu_i \leq \inf  \mu_j$ 
(that is not only $\mu_i(x) \leq \mu_j(x)$, but even  $\mu_i(x) \leq \mu_j(y)$ for any $x, y \in M$).

\end{theorem}

\subsection{Dynamics of $f$} 

${}$

\label{Dynamics}

Define the singular sets $\mathcal S_ \pm= \{x \in M, \lambda(x) = \lambda_ \pm\}$.

\subsubsection{Lyapunov splitting}  \label{Lyapunov}

On $M \setminus (\mathcal S_- \cup {\mathcal S}_+)$, corresponding to the eigenspace decomposition 
of $K= K_f$,   we have a regular and orthogonal 
splitting $TM = E_- \oplus E_+\oplus  E_\lambda$.

 Due  to the relation, 
$f_* K  = \alpha I + \beta K^{-1}$, $f$ preserves this  splitting.

 \begin{remark} Even in the linear situation of a matrix $A \in \GL_d(\R)$, it is rare that $A^*A$ and its conjugate 
 $A^{-1}(A^* A) A$ have the same eigenspace decomposition!

\end{remark} 

\subsubsection{Distortion} Recall  the  definition of the $L$-strength $f_*g(., .)=
\frac{1}{\det K}g(K^{-1}., .)$ vs the ordinary strength $S =  \frac{K^{-1}}{ \det K}$.  

If $y = f(x)$ and 
$u \in T_yM$ belongs to a $\mu = \mu^S(y)$-eigenspace of $S$, then $g_x(D_yf^{-1}u,  D_yf^{-1}u) =  \mu g_y(u, u)$.  In our case,
$D_xf$ sends    $S$-eigenspaces at $x$ to $S$-eigenspaces at $y$, by applying  a similarity of ratio $\frac{1} {\sqrt{ \mu ^S(y)}}$. 

In order to compute this by means of $K$-eigenvalues, observe that 
  $$\det K(x ) =  \lambda_-^{d_-}\lambda_+^{d_+} \lambda(x)$$
 (where $d_-, d_+$ are the respective 
 dimension 
of $E_-$ and $E_+$).

Thus, for  any $x$,  $D_xf$ maps similarly $E_-(x) $ to $E_- (f(x))$ with similarity ratio $\zeta_-(x)$ such that 
$$\zeta_-^2(x) =   {(\det K (f(x)) \lambda_-} =   {(\lambda_-^{d_-} \lambda_+^{d_+} \lambda (f(x)) \lambda_-}$$

As for $D_xf: E_+(x) \to E_+(f(x))$ and $D_xf: E_\lambda(x) \to E_\lambda(f(x))$,  their respective  
 distortions are: 
 $$\zeta_+^2(x) =     {(\lambda_-^{d_-} \lambda_+^{d_+} \lambda (f(x))) \lambda_+} \; \; 
\hbox{and} \; \; \zeta_\lambda^2(x) =     {(\lambda_-^{d_-} \lambda_+^{d_+} \lambda (f(x))) \lambda(f(x))}$$

\subsubsection{Data for $f^{-1}$} Let $\la_1^*, \la_-^*, \la_+^*, \zeta_\la^*, \ldots$ be the analogous quantities corresponding to $f^{-1}$. Observe that $f^{-1}$
preserves the same Lyapunov splitting and thus $\zeta_-^*(x) \zeta_-(f^{-1}(x)) = 1$.  It follows that $\la_1^*(x) = \frac{1}{\la_1 (f^{-1}(x))}$, 
$\la_-^* =   \frac{1}{\lambda_+} $, and  $\la_+^*=   \frac{1}{\lambda_-}$.

\subsubsection{Estimation of the Jacobian}

From above we infer that:
$$(\sf{Jac}\;  f_x)^2 = (\det D_xf)^2 =    (\zeta_-(x) \zeta_\la(x) \zeta_+(x))^2 = 
  ({\lambda_-^{d_-} \lambda_+^{d_+} \lambda(f(x))})^{1+d}$$
Now, in  general,  
 $\Jac f^n_x = \Jac f_{f^{n-1}x} \ldots \Jac f_x $, and hence 
 $$(\Jac f^n_x)^2 
 =   { ((\lambda_-^{d_-} \lambda_+^{d_+})^n (\lambda (f^{n}(x)) \ldots \lambda(f(x)))})^{1+d}$$

\begin{fact} Assume that $A^{-1}\centerdot$ is  decreasing  on  $[\lambda_-, \lambda_+]$, that is $\lambda_+$
is repulsing and $\lambda_-$ is attracting, equivalently, $\la_1$ is  decreasing along $f$-orbits: $\la_1(f(x) \leq \la_1(x)$.  
Then,  
on compact sets $M  \setminus  \mathcal S_+$, 
$(\Jac f^n_x)^2$
is uniformly equivalent to $({\lambda_-^{d_-+1} \lambda_+^{d_+}})^{n(1+d)}$,
when $n \to + \infty$
  (recall that $\mathcal S_+ = \{ x \in M /  \la_1(x) = \la_+ \}$).

\end{fact}

The proof bases on the relation $\lambda(f^k(x)) = (A^{-1})^k \centerdot \lambda(x)$
and the next lemma.

\begin{lemma}  \label{lemma} Let $C $ be a hyperbolic element of $\SL_2(\R)$
with fixed points $\lambda_- <\lambda_+$, with $\lambda_-$ attracting. The sequence 
$$\frac{(C^n\centerdot z) (C^{n-1}\centerdot z) \ldots (C\centerdot z) }{\lambda_-^n}    $$
converges simply in 
  $ [\lambda_- \lambda_+ [$ to a continuous function.
  The convergence is uniform
  in any compact subset of $[\lambda_-, \lambda_+[$.
\end{lemma}

\begin{proof} In a small neighbourhood of $\lambda_-$, the $C\centerdot$-action  is equivalent    to 
a linear contraction fixing $\la_-$,  $h: z \to \alpha (z-\lambda_-)+ \lambda_-$, with 
$0 < \alpha <1$. This equivalence is valid also on any compact interval $[\lambda_- , \lambda_+ - \epsilon]$. 
 Thus $h^nz= \alpha^n(z-\lambda_-) + \lambda_-$. The above 
product is $(c\alpha^n + 1) (c\alpha^{n-1} + 1) \ldots (c\alpha +1)(c+1)$,
where $c = \frac{z-\lambda_-}{\lambda_-}$. This product is convergent since it can be bounded 
by $\Pi_{i=0}^{i=n} (e^{\vert c \vert \alpha^i}) \leq e^{\vert c\vert (\Sigma \alpha^i)}$.
\end{proof}

\begin{corollary}  \label{cor} Keep the assumption $A^{-1} \centerdot $  
decreasing. Then   
$(\lambda_-^{d_-}\lambda_+^{d_+}) \lambda_-  \leq1$ and 
$(\lambda_-^{d_-}\lambda_+^{d_+}) \lambda_+  \geq 1$. (In particular $\la_- < 1 < \la_+$).

\end{corollary}

\begin{proof} By the fact above, if ${\lambda_-^{d_-+1} \lambda_+^{d_+ }} >1$, then 
  $\int_{M  \setminus \mathcal S_-} \Jac f_x^n \to \infty$ when $n \to + \infty$ contradicting  
  that  $M$ has a finite volume. 
  
  The other inequality holds 
  by applying the previous fact  to $f^{-1}$. Observe for this, that indeed the eigen-function  $\la_1^*$
  corresponding to $f^{-1}$ verifies the same decreasing hypothesis. \end{proof} 
 \subsubsection{Justification of the  decreasing hypothesis for  $A^{-1} \centerdot$} Let us see what happens
 if   $A^{-1}\centerdot$ was increasing in $[\la_-, \la_+]$. In this case, the volume estimate would give
 $(\lambda_-^{d_-}\lambda_+^{d_+}) \lambda_+  \leq1$ and 
$(\lambda_-^{d_-}\lambda_+^{d_+}) \lambda_-  \geq 1$, which leads to the contradiction
$\la_+ \leq  \la_-$.

\subsection{The projective Weyl tensor} 
This is a $(3, 1)$-tensor
  $W: TM \times TM \times TM \to TM$, that is 
  invariant under $\Pro(M, g)$:    $W(D_xfu, D_xfv, D_xfw) = D_xf(W(u, v, w))$, for any $u, v, w \in TM$ 
  (and any $f \in \Pro(M, g)$) . In dimension $\geq 3$,  its 
  vanishing is the obstruction to projective flatness of $(M, g)$, that is the fact that $(M, g)$ has a constant sectional curvature. 
 
Unlike  the conformal case, the projective Weyl tensor   is not a curvature tensor, that is it does not satisfy  all the usual 
  symmetries of   curvature tensors  (see for instance \cite{Besse, Eastwood} for more information).  Its true definition is as follows. If 
  $u, v, w, z$ are four vectors in $T_xM$ such that any two of them are either equal or orthogonal (that is they are part of an orthonormal basis), then:
  \begin{equation} \label{definition.Weyl}
g_x(W(u, v, w),   z) = g_x( R(u, v)w, z) - \frac{1}{n-1} ({\delta}^z_v Ric(w, u) - \delta^z_u Ric(w, v) )
\end{equation}
  where $Ric$ is  the Ricci tensor and $\delta$ is the Kronecker symbol.

\subsubsection{Boundedness}  By compactness,  $W$ is bounded by means of $g$, that is $\parallel W(u, v, w) \parallel \leq C \parallel u \parallel \parallel v \parallel \parallel w \parallel $, for some constant $C$,  
where $\parallel . \parallel$ is the norm associated to $g$.

\subsubsection*{Notation $\approx$ } We will deal with sequences of positive  functions $A(n, x)$ and numbers $a_n$. We will say they have the same growth rate for $x $ in a compact set  $K \subset M$,  and write
$A(n, x) \approx a_n$, if the ratio $\frac{A(n, x)} {a_n}$ belongs to an interval $[\alpha, \beta]$, $0 < \alpha  \leq \beta < \infty$, when $x \in K$. In general, $K$ is given by the context and hence omitted.

\subsubsection{Asymptotic growth under $Df$}
\label{growth}
The goal in this  paragraph  is to estimate the asymptotic  behaviour under the tangent dynamics $Df$  of vectors
in each of of the Lyapunov spaces (\S \ref{Lyapunov})

\medskip
\noindent
{\it Case of $E_-$:}
 
$D_xf^n$ maps similarly $E_-(x) $ to $E_-(f^n(x))$
with a  contraction factor 
$$\zeta_-(n, x) = \zeta_-(x) \zeta_-(f(x)) \ldots \zeta_-(f^{n-1}(x))$$ 
More concretely, 
\begin{equation*}
\hbox{ If }  u \in E_-(x), \hbox{then} 
 \parallel D_xf^nu \parallel  =  \zeta_-(n, x) \parallel u \parallel 
\end{equation*}

  Recall  that $\zeta_-(x) =      ({(\lambda_-^{d_-} \lambda_+^{d_+} \lambda (f(x)) \lambda_-})^{1/2}$. If $x \in M \setminus \SS_+$, i.e.
  $\la(x) < \la_+$, then, by Lemma \ref{lemma}, $\zeta_-(n, x) $  grows as 
  $$ \zeta_-(n, x) \approx {(\la_-^{d_-+2} \la_+^{d_+})^{n/2}}, \; n \to + \infty$$

 \bigskip

\noindent
{\it  Case of $E_+$ and $E_\la$:}

  On  defines in the same way $\zeta_+(n, x)$ and $\zeta_\la(n, x)$ as distortion factors of $D_xf^n$ form $E_+(x)$ to $E_+(f^n(x))$ 
 and from $E_\la(x) $ to $E_\la(f^n(x))$),  respectively. One gets when $n \to \infty$
 $$     \zeta_+(n, x) \approx  {(\la_-^{d_-+1} \la_+^{d_++1})^{n/2}} $$
 $$ \zeta_\la(n, x) \approx  {(\la_-^{d_-+2} \la_+^{d_+})^{n/2}} 
 \approx 
 \zeta_-(n,x)$$

In sum, when $n \to + \infty$, 
\begin{eqnarray*}
\hbox{ If }  u \in E_-(x), &
 \parallel D_xf^nu \parallel  =  \zeta_-(n, x) \parallel u \parallel  & \approx {(\la_-^{d_-+2} \la_+^{d_+})^{n/2}} \parallel u \parallel  \\
 \hbox{ If }  u \in E_\la(x), &
 \parallel D_xf^nu \parallel  =  \zeta_\la(n, x) \parallel u \parallel  & \approx {(\la_-^{d_-+2} \la_+^{d_+})^{n/2}} \parallel u \parallel \\
 \hbox{ If }  u \in E_+(x), & 
 \parallel D_xf^nu \parallel  =  \zeta_+(n, x) \parallel u \parallel 
 & \approx  {(\la_-^{d_-+1} \la_+^{d_++1})^{n/2}} 
 \parallel u \parallel 
\end{eqnarray*}

\bigskip

Observe that all these behaviours are uniform on  compact sets of $M \setminus \mathcal S_+$. Also,  observe that, 
surely, $\zeta_-(n,x)$ is exponentially decreasing. Indeed, if not,   $ \la_-^{d_-+2} \la_+^{d_+} \geq 1$, and since $\la_- < \la_+$, 
$\zeta_+(x, n) $  would be exponentially increasing. 
On the other hand, the Jacobian of $Df^n$ is at least equivalent  to  a power of 
$\zeta_+(x, n)$ (since $\zeta_-$ and $\zeta_\la$ are bounded from below).

\subsubsection{A first vanishing}

\begin{fact} For any $x \in M \setminus (\SS_- \cup \SS_+)$, $W(u, v, w) = 0$ once all $u, v, w$ belong
to  $E_-(x) \oplus E_\lambda(x)$, or all belong
to $E_+x) \oplus E_\lambda(x)$.

\end{fact}

\begin{proof}

To begin with, assume  $u, v, w \in E_-(x)$,  then and denote $ z_n= W(D_xf^nu, D_xf^nv, D_xf^nw ) $. By the two previous paragraphs, 
\begin{eqnarray*} \label{all.growth}
\parallel z_n \parallel  & \leq C &   \parallel D_xf^nu \parallel  \parallel D_xf^nv \parallel
 \parallel D_xf^n w \parallel = \\
 \zeta_-(n, x)^3 \parallel u \parallel  \parallel v \parallel \parallel w \parallel   
& \approx & 
 {(\la_-^{d_-+2} \la_+^{d_+})^{3n/2}} C \parallel u \parallel \parallel v \parallel \parallel w \parallel   
  \end{eqnarray*}
That is, 
$$  \parallel z_n \parallel   \lessapprox   {(\la_-^{d_-+2} \la_+^{d_+})^{3n/2}} $$
(where $\lessapprox$ means that the ratio of the left hand by the right one is bounded independently   of $x$).

On the other hand,   by $f$-invariance of $W$,  
 $z_n = W(D_xf^nu, D_xf^nv, D_xf^nw) = D_xf^n z$, 
 for $z= W(u, v, w)$. 
 
 Decompose $z = z^- + z^+ + z^\la  \in E_-(x) \oplus E_+(x) \oplus E_\la(x)$, and  let 
 $z_n = z_n^- + z_n^+ + z_n^\la$
  accordingly, i.e. $z^-_n = D_xf^n z^-$... 
  
  Thus, the norm of each of these parts is dominated by
  $ {(\la_-^{d_-+2} \la_+^{d_+})^{3n/2}} $.
  
  However,  by 
  \S \ref{growth}
  $z^-_n \approx  {(\la_-^{d_-+2} \la_+^{d_+})^{n/2}} \parallel z^- \parallel$. 
  
  If $c$ denotes 
  $ {(\la_-^{d_-+2} \la_+^{d_+})}^{1/2}$, then recall  it  is $<1$. So, we have at the same time
  $z^-_n  \lessapprox c^{3n}$ and $z^-_n \approx c^n \parallel z^- \parallel$, which  obviously implies
  $z^- = 0$.
  
  Vanishing of $z^+$ and $z^\la$ is easier to check. Indeed, they behave like 
  $\zeta_+(n, x)  \parallel z^+ \parallel $ (and $ \zeta_\la(n, x) \parallel z^\la \parallel $ respectively). 
  By \S \ref{growth}, these are at least equivalent to $c^n \parallel z^+ \parallel $
  (and $c^n \parallel z^+ \parallel $ respectively). They can not be dominated by 
  $c^{3n}$, unless $z^+ = z^\la = 0$.
  
  \bigskip
  
  Finally, the case where  all  $u, v$ and $w$ belong to  $E_-(x) \oplus E_\lambda(x)$, or all belong
to $E_+x) \oplus E_\lambda(x)$. can be handled in the same way.
  \end{proof}

\subsubsection{Commutation}   \label{commutation} It was observed in particular in \cite{Kiosak} that any $L \in \L(M, g)$ commutes with the Ricci curvature $Ric$ (seen as an endomorphism of $TM$). From it we deduce that $Ric(u, v) = 0$ when $u$ and $v$ belong to two different 
eigenspaces of $L$ (two commuting symmetric endomorphisms of a Euclidean space have 
a common  eigen-decomposition, which is furthermore orthogonal). More precisely $Ric(u, v) = 0$ when $u \in E_\la(x)$ and $v \in E_\pm(x)$, or
$u \in E_-(x)$ and $v \in E_+(x)$.

\subsubsection{A second vanishing}

\begin{fact} \label{belonging1} Let $u \in E_-(x)$ and $v, w \in E_+(x) \oplus E_\la(x)$, 
then $W(u, v, w) \in E_- (x)$.

 .

\end{fact}

\begin{proof} We have to prove that
 $g_x(W(u, v, w),   z) = 0$, whenever $z \in E_+(x) \oplus E_\la(x)$.  
 
 Recall the defining formula \ref{definition.Weyl} of the Weyl tensor.   Observe that $\delta^z_u = 0$ and $Ric(w, u)= 0$ (because of \S \ref{commutation}), and thus
   $g_x(W(u, v, w),   z) = g_x( R(u, v)w, z)$

   On the other hand
     $g_x(W(w, z, v ),   u) =  g_x( R(w, z)v, u)$  because $\delta^u_w = \delta^u_z = 0$.

     Hence 
     $g_x(W(u, v, w),   z) = g_x(W(w, z, v),    u) $. But, we already proved that $W(w, z, v) = 0$ since all $w, z, v$ belong to 
     $E_+(x) \oplus E_\la(x)$.

\end{proof}

\begin{fact} \label{vanishing2}

In the same conditions, that is, $u \in E_-(x)$ and $v, w \in E_+(x) \oplus E_\la(x)$, we have: $W(u, v, w) = 0$. 

\end{fact}

\begin{proof} We know by above that $z= W(u, v, w)$ belongs to $E_-$. 

By \S \ref{growth}, we can write: $D_xf^n z = \zeta_-(n, x) \overline{z_n}$, $D_xf^n u = \zeta_-(n, x) \overline{u_n}$, where $\overline{z_n}$ and $\overline{u_n}$ belong to $E_-(f^n(x))$, and 
$\parallel \overline{z_n} \parallel = \parallel z \parallel $ and 
$\parallel \overline{u_n } \parallel = \parallel u \parallel$. 

Thus $\overline{z_n} = W(\overline{u_n}, D_xf^n v, D_xf^n w)$. But $D_xf^nv$ and $D_xf^nw$ tend to 0 when
$n \to - \infty$,  hence $\parallel  z \parallel = \parallel z_n \parallel \to 0$, that is $z = 0$.
\end{proof}

\subsubsection{Full vanishing}

${}$

 Similar to the above situation  in  Fact \ref{belonging1}, we consider   the case where $u, v \in E_+(x)  \oplus E_\la(x)$ and $w \in E_-(x) $ and prove that $W(u, v, w) \in E_-(x)$. 
 For this goal,   observe  that  
   $Ric(w, u) = Ric(w, v) = 0$, and hence
     $g_x(W(u, v, w),   z) = g_x( R(u, v)w, z)$. 
     On the other hand, 
      $g_x(W(u, v, z),   w) = g_x( R(u, v)z, w)$, since $\delta^w_u = \delta^w_v = 0$.  Next, $u, v, z \in E_+(x)  \oplus E_\la(x)$ and thus $W(u, v, z) = 0$, and consequently 
     $g_x(W(u, v, w),   z) = 0$ 	as claimed.
     
     \medskip
   Next,  the proof of  Fact \ref{vanishing2} applies here  and yields  that $W(u, v, w) = 0$.  
 
 \medskip
   
   So, by this and the previous facts, 
   we get  that $W(u, v, w) =0$ whenever  (at least) two of them are in $E_+(x) \oplus E_\la(x)$
   (observe that $W(u, v, w)  = - W(v, u, w)$ ). 
     
Obviously, we can switch roles of $E_+$ and $E_-$, and so  get that $W(u, v, w) =0$ in all cases.
$\Box$

 \section{Riemannian case, Proof  Theorem \ref{Riemannian} }
 
  As previously, $(M, g)$ is a compact riemannian manifold with $\Pro(M, g) \supsetneq \Aff(M, g)$. 
  By \cite{Mat1}, the degree of  projective mobility $\dim \L(M, g)$ equals 2. 
 Pick $f \in \Pro(M, g)$ as in \S \ref{homography}.  
 
 If $\rho(f) $ is non-hyperbolic for any choice of such $f$, then, by \S \ref{non.hyperbolic}, 
  $\Pro(M, g)/ \Iso(M, g)$
  is finite as stated in  Theorem \ref{Riemannian}.
  
  If $\rho(f)$ is hyperbolic, then by \S \ref{hyperbolic}, the projective Weyl tensor of $(M, g)$ vanishes.

  \subsection*{End of proof in higher dimension}
  Vanishing of the Weyl tensor in dimension $\geq 3$  means that $(M, g)$ has constant sectional curvature
  (see for instance  \cite{Besse, Eastwood}). The universal cover of $M$ is necessarily the sphere since the Euclidean and hyperbolic  spaces have no projective non-affine transformations.

 \subsection*{Proof in dimension 2} The remaining part of the present section is devoted to the case
 $\dim M = 2$ and $\rho(f)$ hyperbolic.  Our goal in the sequel is to show  that in this case, too, $(M, g)$ 
 has constant curvature.
 
 The spectrum of $K_f$ consists of $\la$, and one constant, say $\la_-$.

 \subsubsection{Warped product structure}
 Let $\mathcal F_-$  and $\mathcal F_\la$
 the two one dimensional foliations tangent to the eigenspaces $E_-$ and $E_\la$. They are regular foliations
 on $M \setminus \SS_-$ (recall that  $\SS_- = \{x / \la(x) = \la_ -\}$).

 Recall that Dini normal form says that  two projectively equivalent metrics $g$  and $\bar{g}$ on a surface have the following form :
   $$g= (X(x) -Y(y) (dx^2 + dy^2), \; \; \bar{g} = (\frac{1}{Y(y)} - \frac{1}{X(x)})(\frac{dx^2}{X(x)} + \frac{dy^2}{Y(y)})$$
in some coordinates system ($x, y)$ and  near any point where the $(1, 1)$-tensor $L$ defined by $ \bar{g}(., .) = g(   \frac{L^{-1}}{\det L} ., .)$,  has simple eigenvalues. In fact, 
 $X(x)$ and $Y(y)$ are  the eigenvalues of $L(x, y)$, and the coordinates are adapted to eigenspaces. 
 
From this  normal form one deduces that $(\mathcal F_\la, \mathcal F_-)$ determines a warped product (for $g$ as well as $\bar{g}$), that is,  
there are coordinates $(r, \theta)$, where $\frac{1}{\partial r}$ (resp. $\frac{1}{\partial \theta}$) is tangent to 
$\mathcal F_\la$ (resp. $\mathcal F_-$), and the metric has the form $dr^2 + \delta(r) d \theta^2$ (see \cite{Zeghib1} for more information on warped products). Indeed here 
$Y(y) = \la_-$, and hence $g = (X(x) - \la_-) dx^2 + (X(x) - \la_-) dy^2$, and then change coordinates according to 
$dr^2 = (X(x) - \la_-) dx^2 $, $\theta = y$.
 
 We deduce in particular that the leaves of $\mathcal F_\la$ are geodesic in $(M, g)$. 
 
 \subsubsection{Topology} By \ref{growth}, $D_xf$ is contracting away from $\SS_+ = \{x / \la(x) = \la_+\}$. 
 
 Let 
 $c \in ] \la_-, \la_+[$  
  and  $M_c = \{x / \la(x) \leq c\}$.  It is a codimension 0 compact submanifold with boundary
 the level $\la^{-1}(c)$, 
  for $c$ generic. In particular, it  has  a finite number of connected components.  
 Since $\la$ is decreasing with $f $ (\S \ref{Dynamics}), $f$ preserves $M_c$:  $f(M_c) \subset M_c$. Taking a power of $f$, we can assume it preserves each component of $M_c$. On such a component, say $M_c^0$,  $f$ contracts the Riemannian metric, and hence also its generated distance. It follows that $f$ has   a unique fixed point $x_0 \in M_c^0$.  The $M_c^0$'s, for 
 $c $ decreasing to $\la_-$, is a decreasing family converging to $x_0$. It follows that 
 these $M_c$'s are topological discs, their boundaries $\la^{-1}(c)$ are circles surrounding $x_0$, and 
 finally $x_0$ is the unique point in $M_c^0$ with $\la(x) = \la_-$.

 It is a general fact on tensors of $\L(M, g)$, an eigenfunction is constant along   leaves tangent to the eigenspaces 
 associated to the other eigen-functions. In our case, 
 $\la$ is constant on the $\mathcal F_-$- leaves, equivalently, leaves of $\mathcal F_-$ are levels of 
 $\la$. So, these leaves are circles surrounding $x_0$.  Those of $\mathcal F_\la$ are orthogonal 
 to these circles, and hence they are nothing but the geodesics  emanating from $x_0$. Thus, 
 $\mathcal F_\la$ and $\mathcal F_-$ have $x_0$ as a unique singularity in 
 $M_c^0$.

  \subsubsection{Geometry} We infer from the above analysis that the polar 
     coordinates around $x_0$ gives rise to a warped product structure, that is, the metric on these coordinates 
 $(r, \theta)$, has the form: $dr^2 + \delta(r) d \theta^2$ (in the general case $\delta$  depends rather  on $(r, \theta)$). 
  
  At $x_0$, $D_{x_0}f$
 is   a similarity  with coefficient $\la_-$.

Observe next that, from the form of the metric, rotations $\theta \to \theta + \theta_0$ are isometries.  Composing 
$f$ with a suitable rotation, we can assume $f$ is a ``pure homothety'', i.e. it fixes individually each 
geodesic emanating form $x_0$.  Thus $f$ acts only at the $r$-level: $f(r, \theta) = f(r)$.

Now, the idea is to construct a higher dimensional example with the same ingredients, e.g. $\delta$ and $f$. Precisely, 
consider the metric $dr^2 + \delta (r) d \Omega^2$, where $d\Omega^2$ is the standard metric on a sphere 
$\mathbb S^N$.  We will show in the  Lemma   \ref{polar_coordinates} below that this metric is indeed smooth

We let $f$ act by $f(r, \Omega) = f(r)$. 

One verifies that $f$ is projective. Indeed, $\SO(N+1)$ acts isometrically and commutes with $f$, and any geodesic is
contained in a copy of our initial surface (since, as in the case of $\R^{N+1}$,  for any two points $x$ and $y$,  one may find a copy of the surface containing them).

This new $f$ has the same dynamical behaviour as the former one, and one proves as in \S \ref{hyperbolic}
that the projective Weyl tensor of this new metric vanishes and it has therefore a constant sectional curvature. 
The same is true for our initial surface.

$\Box$

 \begin{lemma} \label{polar_coordinates}  Consider  a metric $g$ of the form $dr^2 + \delta (r) d \Omega^2$, where 
 $\delta $ is defined on an interval $[0, R[$,  smooth on $]0, R[$,  and $d \Omega^2$ is  the metric of $\mathbb S^N$ (thus $g$ is defined 
 on   a ball $B(0, R) \setminus \{ 0 \}$ in $\R^{N+1} \setminus \{0\}$).  Then $g$ extends smoothly to $0$ if and only if 
 $\delta(r) = \zeta(r^2)$, where $\zeta $ is smooth as a function of   $r$  and $\zeta^\prime(0) = 1$.  In particular $g$ is smooth for some dimension $N >0$  iff it is smooth for any.

 \end{lemma}

\begin{proof}

Observe firstly that the condition on $\delta$ is equivalent to that the function 
 $\eta(r) = \frac{\delta(r) -r^2}{r^4}$  equals  $\kappa(r^2)$ where $\kappa$ is smooth on $r$.

Consider 
 the mapping  $\Omega: z = (z_1, \ldots, z_{N+1}) \in \R^{N+1} \setminus \{0\} \to  \Omega (z) = \frac{z}{\parallel z \parallel} = (\Omega_1, \ldots, \Omega_{N+1}) \in \mathbb S^N$. So $d \Omega=  (d \Omega_1, \ldots, d \Omega_{N+1}) $ is a vectorial 1-differential form on 
$\R^{N+1}\setminus \{0\}$ and 
 $d\Omega^2 = \Sigma d \Omega_i^2$ is a field of quadratic forms (on $\R^{N+1} \setminus \{0\}$) whose 
restriction to $\mathbb S^N$ coincide with the induced metric. 

Similarly, $r = \parallel z \parallel$, and hence  
$dr = \frac{1}{r}( \Sigma z_idz_i)$.  Thus $r^2dr^2 =  (\Sigma z_idz_i)^2$.

On the other hand,  it is  known that $g_E= r^2 + r^2 d \Omega^2$ is the   Euclidean metric $\Sigma dz_i^2$. Therefore $r^4 d\Omega^2$ is smooth and  equals exactly:
 $$r^4 d\Omega^2 =  r^2(g_E - dr^2)=  (\Sigma z_i^2)(\Sigma dz_i^2) - (\Sigma z_i dz_i)^2 =  \Sigma_{i \neq j} z_i^2 dz_j^2 - z_i z_j dz_i dz_j$$

 Now, let $g = dr^2 + \delta (r) d \Omega^2$. Then 
 $$g-g_E =  (dr^2 + \delta(r) d \Omega^2) - (dr^2 + r^2 d\Omega^2)= (\delta(r) -r^2) d \Omega^2 =
 \frac{(\delta(r) -r^2)}{r^4} (r^4 d \Omega^2 ) = \eta (r) (r^4 d \Omega^2)$$
 
 We deduce in particular that  a sufficient condition for $g$ to be smooth  (as a function of $z$) is that $\eta$ is smooth
 (as a function of $z$). 
 To see that this is also   a necessary condition, we infer from the previous  formula for $r^4 d \Omega^2$ that
 $\eta (r) r^4 d\Omega^2$ is smooth iff the functions  $\eta(r) z_iz_j$ are  smooth for any $i, j$. Then apply the next lemma:
 
 \begin{lemma} \label{smooth_functions} Let $\eta(r)$ a function such that all the functions  $\eta(r)z_i z_j$ are smooth on $z$. Then 
 $\eta(r)$ is smooth (as a function of $z$) and it equals  $\kappa(r^2)$ where $\kappa$ is smooth as a function of $r$.
 
 \end{lemma}
 
 \begin{proof}

 First, $\eta (r)z_i^2 \to 0 $ when $r \to  0$. Indeed if not, this limit does not depend on $i$, and one can take   the ratio 
 $ 1 = \lim \frac{\eta(r) z_i^2} {\eta(r) z_j^2} =  \lim \frac{z_i^2}{z_j^2}$, but the latter limit does not exist.

 For the next step, to simplify notations, let us assume the  dimension is 2 and note $x = z_1, y = z_2$
 (the proof in higher dimension is identical).
 
  By hypothesis $T(x, y)=  \eta(r) (x^2 + y^2) = \eta (r) r^2$ is smooth. 
Its  Taylor expansion  allows one to write it, up to any order,  as a sum of homogeneous polynomials on $x$ and $y$.
 Since
 $T$ is $\SO(2)$-invariant, the same is true for these polynomials. 
 Now, let $P$ such a polynomial of degree $k$.  By homogeneity and 
 $ \SO(2)$-invariance $\frac{P} {(x^2 + y^2)^{\frac{k}{2}}}$ is constant on $\mathbb S^1$
 and hence constant, that is $P$ is proportional to $r^{\frac{k}{2}}$. This implies in particular that $k$ is even.  Therefore, the Taylor expansion is on the powers $r^2, r^4, r^6, \ldots$.  We finally get that $\eta(r)$ has a Taylor expansion on $1, r^2, r^4, \ldots$..., which ensures the existence of $\kappa$. 
 \end{proof}

 \bigskip

 Coming back to the proof of Lemma  \ref{polar_coordinates}, 
assume now that 
   $g = dr^2 + \delta (r) d \Omega^2$ is a  smooth Riemannian metric on a neighbourhood of  0.  For fixed $\Omega$, the ray $r \to (r, \Omega)$
is an arc-length parameterized geodesic. It follows that $(r, \Omega) \to z = r\Omega$ are normal coordinates, i.e. the inverse
of the exponential map at 0. Thus,  as in  the Euclidean  case,  the metric $g$ is smooth with respect to $z$. 
Therefore, $\delta(r) $ satisfies the same conditions as above. 
\end{proof}

\begin{remark} As an alternative for all this proof in dimension 2, the referee suggests to mimic the higher dimensional proof
by replacing the Weyl tensor by  its 2-dimensional version, the Liouville tensor (as used in  \cite{{Bryant}}).

\end{remark}

 \section{The K\"ahler case: Proof of Theorem \ref{K\"ahler}}
 \label{K\"ahler.case}

 Let $F(N, b)$  denote the simply connected Hermitian space of dimension $N$ and constant holomorphic sectional curvature 
 $b$.  Calabi proved (in his thesis) the following striking fact:
 
 \begin{theorem}[Calabi \cite{Calabi}] Let $M$ be a K\"ahler manifold (not necessarily complete) and $f: M \to F(N, b)$ 
 a holomorphic isometric immersion.
  Then, $f$ is rigid in the sense that any other  immersion $f^\prime$ is deduced from $f$ by 
 composing with an element of $\Iso(F(N, b))$  
 (this element is unique if the image of $f$  in not contained in a totally geodesic proper  subspace  
 of $F(N, b)$). In particular, $f$ is equivariant with respect to  some faithful representation 
  $\Iso(M) \to \Iso(F(N, b))$.
 
  \end{theorem}
 
 As for holomorphic isometric immersions between space forms, one deduces  (for more information, see   for instance in \cite{Takeuchi, Hul, Discala}):
 
\begin{theorem} --  The  K\"ahler Euclidean  space $\C^d$ can not embed holomorphically isometrically in a projective space $\mathbb P^N(\C)$ (a radially simple example of this is the situation of a holomorphic vector field; it  does not act isometrically, in particular its orbits are not 
metrically homogeneous).

--   Up to ambient isometry,  the
holomorphic homothetic  embeddings between projective spaces are given by Veronese maps: 
$v_k: (\mathbb P^d(\C), g_{FS}) \to  (\mathbb P^N(\C), \frac{1}{k} g_{FS})$, $N = \binom{d+k}{k}-1$, $v_k: [X_0, \ldots, X_d]  \to [ \ldots X^I \ldots ]$, where $X^I$ ranges  over all monomials of degree $k$ in $X_0, \ldots, X_d$. 

\end{theorem}
 
 \bigskip

 \subsubsection*{Proof of Theorem \ref{K\"ahler}} This will follow from 
our 
rigidity theorem of the $h$-projective group of 
 K\"ahler manifolds (see \S \ref{Kaehler}),  together with the following fact.
 
 \begin{fact} Let $(M^d, {g_{SF}}_{\vert M}) $ be a submanifold of $(\mathbb P^N(\C), g_{SF})$,  then
 $\Aff(M^d, {g_{SF}}_{\vert M})/ \Iso (M^d, {g_{SF}}_{\vert M}) $ is finite (vaguely bounded by $ \frac{n}{2} !$).
 
 \end{fact}
 
 \begin{proof} This is a standard idea (see for instance \cite{Koba, Zeghib}), the unique special fact we use here is that,   by Calabi Theorem, the universal cover has no flat factor in its De Rham decomposition. 
 Thus $\tilde{M}$ is a product $ \tilde{M_1} \times \ldots \times \tilde{M_m}$ of irreducible K\"ahler  manifolds. 
 
 The holonomy group
 $Hol^{\tilde{M} }$ equals the product $Hol^{\tilde{M_1} } \times   \ldots \times Hol^{\tilde{M_m} }$. An affine transformation $\tilde{f} $ commutes with $Hol$ and hence preserves the De Rham splitting. 
 Taking a power, we can assume that $\tilde{f}$ actually  preserves each factor, and we will  thus prove that 
 $\tilde{f} $ is isometric.  Since $\tilde{M_i}$ is irreducible, $\tilde{f}$ induces a homothety 
 on it, say of distortion $c$. If $c \neq 1$, then $f$ or $f^{-1}$ is contracting  with respect to the distance 
 of $\tilde{M_i}$. 
 In this case, $\tilde{f}$ will have a (unique)  fixed point in $\tilde{M_i}$. However, $\tilde{f}$ preserves the
 Riemann curvature tensor $R(X, Y)Z$ of 
 $\tilde{M_i}$.  But being invariant by a contraction (or a dilation), this tensor must vanish, that is $\tilde{M_i}$ is flat, contracting the fact that it is irreducible.  Therefore, $c = 1$, that is $\tilde{f}$ is isometric.
 
  \end{proof}
 
 \begin{remarks} ${}$
 
 1. By equivariance,  Segre maps $\mathbb P^m (\C)\times \mathbb P^n (\C)  \to \mathbb P^{(m+1)(n+1)-1} (\C)$ are homothetic. 
 In particular, some $(\mathbb P^1 (\C) \times \mathbb P^1 (\C), \frac{1}{k} (g_{FS} \oplus g_{FS}))$ can be embedded in some
 $\mathbb (P^N(\C), g_{FS})$. By composing Veronese and Segre maps, one can also realize some 
 metrics $(\mathbb P^1 (\C),  \frac{1}{k} g_{FS}) \times  (\mathbb P^1 (\C),  g_{FS}))$.
 
2.  In fact, it turns out that for $M$ a submanifold of $\mathbb P^N(\C)$, De Rham decomposition applies to 
$M$ itself; that is the splitting of $\tilde{M}$ descends to a one of $M$. I am indebted to A. J.  Di Scala for giving me
 a proof of that  using Calabi rigidity.  Indeed, this rigidity has the following amuzing corollary:   if $M$ is holomorphically isometrically embedded  in 
$\mathbb P^N(\C)$, then neither a cover nor a quotient of it can be embedded  so. Now, the De Rham splitting of $\tilde{M}$ gives 
an immersion into products of projective spaces. Segre map is isometric form this product to one big projective space, which gives us another 
holomorphic isometric immersion of $\tilde{M}$. But this must coincide with the immersion given by the universal cover $\tilde{M} \to M$. This 
implies that the  De Rham decomposition is defined on $M$ itself.

  \end{remarks}

 \section{Facts on the  indefinite pseudo-Riemannian case: Proof of Theorem \ref{pseudo}}
 
Let $(M, g)$ be a compact pseudo-Riemannian manifold with projective degree 
of mobility $\dim \L(M, g) = 2$, such that $\Pro(M, g)/\Aff(M, g)$ is infinite. 
Consider  $\rho: \Pro(M, g) \to \GL_2(\R)$. 

Denote  $G= \rho(\Pro(M, g))$. 
Theorem
\ref{pseudo} says that, up to  finite index, $\ker \rho =  \Iso(M, g) = \Aff(M, g)$,  and $G$ 
 lies in a non-elliptic one parameter group.

   \subsection{``Projective linear'' action of $\Pro(M, g)$}

 So far, we singled out  an element $f \in \Pro(M, g)$ and associate to it a homography $A\centerdot$ acting on $\bar{\R}$. It turns out 
 that this $A\centerdot$ is nothing but   that corresponding 
  to the (projective)  action of $\rho(f) $  on the projective space of $\L (M, g)$, identified to $\mathbb P^1(\R)$, via the basis $\{K= K_f, I\}$.  

Indeed, the choose of the basis $\{K, I\}$, say co-ordinates $(k, i)$,  
  allows one to identify  $\mathbb  P (\L(M, g))$ with $  \mathbb P^1 (\R)$. In the affine chart $[k: i] \in   \mathbb P^1(\R) \to z = \frac{k}{i}$, 
 the projective action of  $\rho (f) $  
 is $z \to  \frac{\alpha z + \beta}{z}$, where  $\alpha$ and $\beta$ are  defined by $\rho(f) K = \alpha K + \beta I$, as in 
  \S \ref{homography}.

   Now, we let the whole group $\Pro(M, g)$ act by means of $\rho$ on the projective space, and in fact the complex one. 
More precisely, let  $$\Phi: \Pro(M, g) \to \PGL(\L(M, g) \otimes \C)$$ be the action associated to $\rho$ on $\mathbb P (\L(M, g) \otimes \C)$, the projective space of the complexification  
of $\L(M, g)$. 

The degeneracy set $\mathcal D$ is complexified as $$\mathcal D^\C =  \{L \in \mathbb P^1(\L(M, g) \otimes \C), L \; \hbox{not an isomorphism of }\;  TM \otimes \C \}$$
   
   The proof of the following fact is similar to that of Fact \ref{sectors}.
 
  \begin{fact}   Let $f $ be any element of $ \Pro(M, g)$ with $K_f \neq - \pm I$, then $\mathcal D^\C$ can 
  be computed by means of $K_f$ as follows.   Under the identification 
of $\mathbb P (\L(M, g) \otimes \C)$ with $\mathbb P^1(\C)$ via the basis $\{K_f, I\}$, the set $\mathcal D^
\C$ corresponds  to the range of the spectrum mapping of $K_f$:  $x \in M \to Sp^{K_f}(x) =$ Spectre of $K_f(x) \subset \C$.

  \end{fact}

  The point is that this set is invariant under the $G$-action.

  \subsubsection{} By Fact  \ref{injective}, the projection of $G$  on its image in $\PGL_2(\R)$ has finite index. In fact, since we are interested 
 in  objects up to finite index, for simplicity seek, we will argue as if $G$ is contained in 
   $   \PGL_2(\R)$, in fact in $\SL_2(\R)$ to be more concrete.

\subsubsection{The Kernel of $ \rho$} Let $h \in \Aff(M, g)$,   we will prove that $\rho(h) = 1$, up to index 2. Since $h$
is affine,  all $K_h$-eigenvalues are constant.  It follows that $K_h$ has the form $a I$, since otherwise it generates together with $I$ the whole $\L(M, g)$, and hence all the $K_f$ will have constant eigenvalues for any $f$,  contradicting the fact 
that 
$\Pro(M, g) \supsetneq \Aff(M, g)$. By finiteness of the volume, $a = \pm 1$, say $a = 1$, i.e. $K_h = I$. 
Now, 
$\rho(h) L = h_* L K_h = h_*L$,   
and thus $\rho(h)I= I$.  
 Therefore, if $\rho(h) \neq 1$, $\rho(h)$ will be parabolic with unique fixed point $I$ (in $\mathbb P^1(\C)$).
 So, any closed $\rho(h)$- invariant set contains $I$. But this is not the case of the degeneracy set 
  $\mathcal D^\C$ (since it corresponds to the spectrum).

\subsection{Proof that $G$ is contained in a one parameter group} As suggested by the referee, we will make use 
of Theorem 1.11 of \cite{Bol-Mat_2}. It states that if for some $x$, $K_f(x) $ has a complex eigenvalue $\lambda$,  then this is a constant eigenvalue, that is, $\lambda$ is eigenvalue of  $K_f(y) $ for any $y \in M$.  So the proof will be essentially similar to the Riemannian case.  More precisely, let $f \in \Pro(M, g)$ such that $\rho(f)$ is hyperbolic or parabolic,  then $K_f$ has everywhere a real spectrum.  Indeed, otherwise, the homography associated to $f$ will have a non-real 
fixed point in $\mathbb P^1(\C)$ which is impossible (since this homography is real). 

Furthermore,  the range of the spectrum of  $K_f$ is  a compact interval in $\R$ (non-reduced to a point since $f$
is not affine). Now,  a parabolic homography preserves no non trivial compact interval, and so this case is impossible.
In the hyperbolic case, the unique  non-trivial invariant interval is that joining the two fixed points.  It follows that 
$\mathcal D^\C$ is an interval in $\mathbb P^1(\R) \subset \mathbb P^1(\C)$. Therefore, the group $G$ preserves a subset of two points consisting in the extremities of this interval. But the subgroup of $\SL_2(\R)$
preserving two points in $\mathbb P^1(\R)$ has a one parameter subgroup as a normal subgroup of index two (e.g.  in 
the case of $\{0, \infty\}$, this group is generated by  of $z \to a z$, $a >0$, and $z \to \frac{1}{z}$).

\subsubsection{Elliptic case} 
It remains now to consider the case where all the elements of $G$ are elliptic, the goal here is to prove that $G$ is finite.  Let $\bar{G}$ be the closure of $G$ and $\bar{G}^0$ its identity component.  Thus  $\bar{G}^0$ is a connected 
subgroup of $\SL_2(\R)$.  It can not be $\SL_2(\R)$ since the set of elliptic elements there is not dense. 
The  4 others possibilities for  non-trivial connected subgroups are, up to conjugacy:  the affine group $\Aff(\R)$ 
(upper triangular elements of $\SL_2(\R)$ or a one parameter of hyperbolic, parabolic or elliptic type. But, the set of elliptic elements is dense (actually just non-trivial) only in the case of an elliptic one parameter group.  Hence, if non-trivial, $\bar{G}^0$ is conjugate to $\SO(2)$. The group $G$ itself is contained in the normalizer of $\bar{G}^0$ which also 
equals $\SO(2)$.   We will now find a contradiction leading to that this situation is impossible.  Indeed, since $G$ is dense in $\SO(2)$, its orbits in 
$\mathbb P^1(\R)$ are dense, and hence any  $G$-invariant closed set in $\mathbb P^1(\R)$
equals $\mathbb P^1(\R)$.  This implies that $\mathcal D^\C \cap \mathbb P^1(\R) = \emptyset$, since this  closed $G$-invariant subset that does not contain $\infty$. In sum, the spectrum of $K_f$ is 
nowhere real. As above, by \cite{Bol-Mat_2}, this implies $K_f$ has a constant spectrum and hence $f$ is affine, but we have already excluded this possibility.

Let us consider now the case where $\bar{G^0} = 1$  which means that $G$ is discrete. Any element $A \in G$
is elliptic and hence conjugate to an element of $\SO(2)$ which has a finite order (by discreetness).   Apply   Selberg Lemma (see for instance \cite{Roger}), which says that a finitely generated subgroup
of $\GL_n(\R)$ has a torsion free finite index subgroup (i.e. with no elements of finite order). 
Let $G^\prime$ be a finitely generated subgroup of $G$. Since all elements of $G^\prime$ have finite order, Selberg Lemma
implies that $G^\prime $ is finite.  However, a finite  non-trivial  subgroup of $\SL_2(\R)$ is conjugate  to 
a unique one parameter elliptic subgroup (geometrically, it a has a unique fixed point in the hyperbolic plane). Say, if an element $A \in G$, up to conjugacy belongs to 
  $\SO(2)$, then, for any $B \in G$, the group $G^\prime$ generated by $A$ and $B$ must be contained in $\SO(2)$, and therefore $G \subset \SO(2)$. As above, $G$  can not be dense in $\SO(2)$ and is hence finite. 
  
\bigskip

    We have thus proved that in all cases and after neglecting finite index objects, $\rho(\Pro(M, g))$ lies in a  hyperbolic or parabolic one parameter group, 
  which completes the proof of Theorem \ref{pseudo}.
   $ \Box$

\begin{remark} In higher dimensions, i.e. for subgroups of $\SO(1, n), n >2$, it is not longer true that having all its elements elliptic implies
the subgroup is contained in a compact subgroup, 
see \cite{Waterman}

\end{remark}

\begin{remark}  Let $P$ the one parameter group that contains $G$
(up to finite things). 
Then $G$ may be equal to $P$, or    dense (and $\neq P$), or finally  discrete and hence cyclic generated
by a single element. The case $G = P$ means that $M$ has a projective vector field. 
One may  ask if the dense case may happen, that is if $G$ is dense, then
necessarily $G = P$?

\end{remark}

 \end{document}